\documentclass[11pt]{amsart}
\usepackage{graphicx}
\usepackage{amsmath, mathtools}
\usepackage{amsfonts, amssymb, color}
\usepackage{verbatim} 
\usepackage[all]{xy}
\usepackage{hyperref}
\usepackage{enumerate}
\usepackage[noadjust]{cite}

\usepackage{array}
\newcolumntype{P}[1]{>{\centering\arraybackslash}p{#1}}

\newcommand{\mult}{\mathrm{mult}}
\newcommand{\SL}{\mathrm{SL}}

\newcommand{\oo}{\mathcal{O}}

\newcommand{\PP}[0]{\mathbb{P}}
\renewcommand{\geq}{\geqslant}
\newcommand{\ol}[0]{\overline{\lambda}}

\newtheorem{proposition}{Proposition}[section]
\newtheorem{theorem}[proposition]{Theorem}
\newtheorem{lemma}[proposition]{Lemma}

\theoremstyle{definition}

\newtheorem{definition}[proposition]{Definition}
\newtheorem{remark}[proposition]{Remark}

\title{Moduli of cubic surfaces and their anticanonical divisors}
\author[Gallardo]{Patricio Gallardo}
\email{pgallardocandela@wustl.edu}
\address{Department of Mathematics,\\
	Washington University in St. Louis}
\author[Martinez-Garcia]{Jesus Martinez-Garcia}
\email{J.Martinez.Garcia@bath.ac.uk}
\address{	Department of Mathematical Sciences\\
	University of Bath}

\begin{document}

\maketitle
\begin{abstract} 
	We consider the moduli space of log smooth pairs formed by a cubic surface and an anticanonical divisor. We describe all compactifications of this moduli space which are constructed using Geometric Invariant Theory and the anticanonical polarization.  The construction depends on a weight on the divisor.  For smaller weights the stable pairs consist of mildly singular surfaces and very singular divisors. Conversely, a larger weight allows more singular surfaces, but it restricts the singularities on the divisor. The one-dimensional space of stability conditions decomposes in a wall-chamber structure. We describe all the walls and relate their value to the worst singularities appearing in the compactification locus. Furthermore, we give a complete characterization of stable and polystable pairs in terms of their singularities for each of the compactifications considered.

\end{abstract}
\section{Introduction}

The moduli space of cubic surfaces is a classic space in algebraic geometry. Indeed,  its GIT compactification was first described by Hilbert in 1893  \cite{hilbert}, and  several alternative compactifications have followed it (see  \cite{ishii1982moduli,naruki1982cross,hacking2009stable}).   In this article, we enrich this moduli problem by parametrizing pairs $(S,D)$ where  $S \subset \PP^3$ is a cubic surface, and  $D \in |-K_S|$ is an anticanonical divisor. There are several motivations for our construction. Firstly, it was recently established that the GIT compactification of cubic surfaces corresponds to the moduli space of $K$-stable  del Pezzo surfaces of degree three \cite{Odaka-Spotti-Sun}.  The concept of $K$-stability has a natural generalization to log-K-stability for pairs, and our GIT quotients are natural candidates to construct compactifications of log K-stable pairs of cubic surfaces and their anticanonical divisors. Therefore, our description is a first step toward a generalization of \cite{Odaka-Spotti-Sun} to the log setting, an approach considered in the sequel to this article \cite{gallardo-margar-spotti}.
Secondly,  a precise description of the GIT of cubic surfaces is important for describing the complex hyperbolic geometry of the moduli of cubic surfaces, and constructing new examples of ball quotients (see \cite{allcock2002complex}).
 More specifically, Laza et al. \cite{Laza-Pearlstein-Zheng} predicted a Hodge theoretical compactification of  the moduli space of pairs $(S,D)$ via a particular loci within the moduli space of cubic fourfolds. 
  One may expect that such uniformization coincides with one of the compactifications of the moduli space of pairs $(S,D)$ that we obtain in this article.  Finally, our compactifications explore the setting of variations of GIT quotients for log pairs for which few examples exists (see \cite{laza2009deformations} and \cite[Theorem 11.2]{dolgachev}).

	The GIT quotients considered depend on a choice of a linearization $\mathcal{L}_t$ of the parameter space $\mathcal H$ of cubic forms and linear forms in $\mathbb P^3$. We have that $\mathcal H\cong \mathbb P^{19}\times \mathbb P^3$. Every ample divisor in $\mathrm{Pic}(\mathbb P^{19}\times \mathbb P^{3})\cong \mathbb Z\langle a \rangle \oplus \mathbb Z\langle b \rangle$ is of bidegree $(a,b)$ for some positive integers $a$ and $b$. Thus, the different GIT quotients arising by picking different polarizations of $\mathcal H$ are controlled by the parameter $t=\frac{b}{a}\in \mathbb Q_{>0}$ (see Section \ref{sec:GIT-setup} for a thorough treatment). For each value of $t$,  there is a  GIT compactification  $\overline{M(t)}$ of the moduli space of pairs $(S,D)$ where $S$ is a cubic surface and $D \in |-K_S|$ is an anticanonical divisor.   It follows from the general theory of variations of GIT (see  \cite{thaddeus1996geometric}, \cite{dolgachev1998variation}, c.f. \cite[Theorem 1.1]{VGITour}) that $0\leqslant t\leqslant 1$ and that there are only finitely many different GIT quotients associated to $t$.   Indeed, there is a \emph{set of chambers} $(t_i,t_{i+1})$ where the GIT quotients $\overline{M(t)}$ are isomorphic for all $t\in (t_i, t_{i+1})$, and there are finitely many \emph{GIT walls} $t_1,\ldots, t_{k}$ where the GIT quotient is a birational modification of $\overline{M(t)}$ where $0<|t-t_i|<\epsilon\ll 1$. Additionally there are initial and end walls $t_0=0$ and $t_{k+1}=1$.

\begin{lemma}
\label{lemma:walls}
The GIT walls are
$$t_0=0,\ t_1=\frac{1}{5},\  t_2=\frac{1}{3},\ t_3=\frac{3}{7},\  t_4=\frac{5}{9},\ t_5=\frac{9}{13},\ t_6=1.$$
\end{lemma}

Given $t\in \mathbb Q_{>0}$ we say that a pair $(S,D)$ is \emph{$t$-stable} (respectively \emph{$t$-semistable}) if it is $t$-stable (respectively $t$-semistable) under the $\SL(4,\mathbb C)$-action. A pair is \emph{strictly $t$-semistable }if it is $t$-semistable but not $t$-stable. The space $M(t)$ parametrizes $t$-stable pairs and $\overline{M(t)}$ parametrizes closed $t$-semistable orbits. 

The GIT walls can be interpreted geometrically as follows. Let $T$ be one of the possible isolated  singularities in a cubic surface (see Proposition \ref{proposition:singcub}),  let $w(T)$ be the sum of its  associated weights (see Definition \ref{def:ADEweights}). For example, the set of weights for the $\boldsymbol{A}_n$ singularity  is 
$\left( \frac{1}{2}, \frac{1}{2}, \frac{1}{n+1}  \right)$ and $w(\boldsymbol{A}_n)=\frac{n+2}{n+1}$.  We define $\mathrm{Wall}(T)\coloneqq\frac{4}{w(T)}-3$.
\begin{theorem}\label{theorem:GITwalls}
	There are $13$ non-isomorphic  GIT quotients $\overline{M(t)}$. Seven of these quotients correspond to the walls $t_i$ in Lemma \ref{lemma:walls} and they can be recovered as $t_i=\mathrm{Wall}(T)$ for each isolated ADE singularity $T$ occurring at a point $p$ of an irreducible cubic surface $S$ such that
	\begin{enumerate}[(i)]
			\item for some $D\in| -K_S|$, the log pair $(S,D)$ is strictly $t_i$-semistable,
			\item for $t'<t_i$ $(S,D)$ is $t'$-unstable, and
			\item $p\not\in \mathrm{Supp}(D)$ unless $t=0$.
	\end{enumerate}
	Indeed, the values of the walls are:
	\begin{align*}
	t_0&=\mathrm{Wall}(\boldsymbol{A}_2)=0, &t_1=\mathrm{Wall}(\boldsymbol{A}_3)= \frac{1}{5},\quad  &t_2=\mathrm{Wall}(\boldsymbol{A}_4)=\frac{1}{3}, \\
	t_3&=\mathrm{Wall}(\boldsymbol{A}_5)=\mathrm{Wall}(\boldsymbol{D}_4)= \frac{3}{7}, &t_4=\mathrm{Wall}(\boldsymbol{D}_5)= \frac{5}{9},\quad &\\
	t_5&=\mathrm{Wall}({\boldsymbol{E}}_6)=\frac{9}{13}, & t_6=\mathrm{Wall}(\widetilde{\boldsymbol{E}}_6)=1.\quad &
	\end{align*}
	The other six GIT quotients $\overline{M(t)}$ correspond to  linearizations $t\in (t_i, t_{i+1})$, $i=1,\ldots,6$. All the points in $\overline{M(t_0)}$ and $\overline{M(t_6)}$ correspond to strictly semi-stable pairs, while all other $\overline{M(t)}$ with $t\in(0,1)$ have stable points. The GIT quotient is empty for any $t \not\in [0,1]$.
\end{theorem}
We will learn in Section \ref{sec:Classification} that the walls $t_i$ and classification of the log pairs $(S,D )$ parametrized by $\overline{M(t)}$ depend on both the singularities of the surface and the divisor $D$ in a complementary way. Indeed, the singularities of the surfaces will be worse when $t$ approaches $1$ while the singularities of the hyperplane section will be worse when $t$ approaches $0$ (see Table \ref{tab:VGITcubics}).

Furthermore, we have a complete analysis of the stability of pairs $(S, D)$ represented in $\overline {M(t)}$ and $M(t)$ for each $t$ in the space of stability conditions $[0,1]\cap \mathbb Q$. Specifically for each $t\in (0,1)\cap \mathbb Q$, in Theorem \ref{theorem:stableGIT} and Table \ref{tab:VGITcubics} we give a list of all $t$-stable pairs represented in $M(t)$, and in Theorem \ref{theorem:GITb} and Table \ref{tab:VGITboundary}, we classify all strictly $t$-semistable pairs with close orbits, which compactify $M(t)$ into $\overline{M(t)}$.  
The quotient $\overline{M(0)}$ is isomorphic to the GIT of cubic surfaces and the quotient $\overline{M(1)}$ is the GIT of plane cubic curves (see \cite[Lemma 4.1]{VGITour}). These spaces are classical and they are described in \cite[Sec 7.2(b)]{mukai} and \cite[Example 7.12]{mukai} respectively.
Henceforth we will focus on the case $t\in (0,1)$. As mentioned earlier, the following theorem gives a first approximation to the classification of log stable pairs of other stability theories, in particular for log K-stability (and the existence of K\"ahler-Einstein metrics with conical singularities along a boundary). This was first observed for cubic surfaces (no boundary) by Ding and Tian in \cite{Ding-Tian-YTDconjecture}.

\subsection*{Notation used, structure of the article and acknowledgements}
Throughout the article a pair $(S, D)$ consists of a cubic surface $S\subset\mathbb P^3_{\mathbb C}$ and an anticanonical section $D\in |-K_S|\cong \mathbb P(H^0(S, \mathcal O_S(1)))$ Hence, $D=S\cap H$ for some hyperplane $H=\{l(x_0,\ldots,x_{3})=0\}\subset \mathbb P^3_{\mathbb C}$. Whenever we consider a parameter $t\in (t_{i}, t_{i+1})$ we implicitly mean $t\in (t_{i}, t_{i+1})\cap \mathbb Q$.

In Section \ref{sec:GIT-setup} we describe in detail the GIT setting we consider. We introduce the required singularity theory in Section \ref{sec:sing}.   GIT-stability depends on a finite list of geometric configurations characterized in Section \ref{sec:1ps}. We  prove Theorem \ref{theorem:stableGIT} in Section \ref{sec:stable}.  We prove Theorems \ref{theorem:GITwalls} and \ref{theorem:GITb}  in Section \ref{sec:stable}. 

Our article does an extensive use of J.W. Bruce and C.T.C. Wall's elegant classification of singular cubic surfaces \cite{classificationcubics} in the modern language of Arnold.  Our results use some computations done via software. The computations, together with full source code written in Python can be found in \cite{gallardo-jmg-code}. The code is based on the theory developed in our previous article \cite{VGITour} and a rough idea of the algorithm can be found there.
The source code and data, but not the text of this article, are released under a Creative Commons CC BY-SA $4.0$ license. See \cite{gallardo-jmg-code} for details. If you make use of the source code and/or data in an academic or commercial context, you should acknowledge this by including a reference or citation to \cite{VGITour} --- in the case of the code --- or to this article --- in the case of the data.

P. Gallardo is supported by the NSF grant DMS-1344994 of the RTG in Algebra, Algebraic Geometry, and Number Theory, at the University of Georgia. This work was completed at the Hausdorff Research Institute for Mathematics (HIM) during a visit by the authors as part of the Research in Groups project \emph{Moduli spaces of log del Pezzo pairs and K-stability}. We thank HIM for their generous support. The final version of the article was completed while the second author was a visitor of the Max Planck Institute for Mathematics in Bonn. He thanks MPIM for their generous support.

We thank Radu Laza and Cristiano Spotti for useful discussions. We thank an anonymous referee for many suggestions which have improved the presentation considerably.
\section{Classification of stable orbits and compactification log pairs}
\label{sec:Classification}

A nice feature of $M(t)$ is that for each $t\in (0,1)$ and each $t$-stable pair $(S,D)$, the surface $S$ has isolated ADE singularities. The classification is simplified by using the notion of `worse singularity'. Roughly speaking, a singularity germ $T_1$ is worse than a singularity $T_2$ if the former can be partially deformed into the latter. See Definition \ref{def:dgd} and Figure \ref{fig:adj} for a formal definition. Table \ref{tab:VGITcubics} gives a summary of the $t$-stable pairs $(S,D)$ for each $t$ in terms of their \emph{worst} singularities and the intersection of the components of $D$.

\begin{table}[!h]\label{tab:VGITcubics}
	\begin{center}
		\begin{tabular}{ |p{1cm}||p{2.1cm}|p{2.1cm}|p{2.1cm}|p{2.1cm}|}  \hline
			$t$ & $(0,\frac{1}{5})$ & $\frac{1}{5}$& $(\frac{1}{5},\frac{1}{3})$ & $\frac{1}{3}$\\  
			\hline
			{$\mathrm{Sing}(S)$ } & $\boldsymbol{A}_2$ & $\boldsymbol{A}_2$& $\boldsymbol{A}_3$  &$\boldsymbol{A}_3$ \\
			\hline
			{$\mathrm{Sing}(D)$ } & on smooth or $\boldsymbol{A}_1\in S$ & isolated on smooth or $\boldsymbol{A}_1\in S$ & isolated on smooth or $\boldsymbol{A}_1\in S$ & isolated or cuspidal at $\boldsymbol{A}_1\in S$ 
			\\
			\hline\hline
			$t$ & $(\frac{1}{3},\frac{3}{7})$ & $\frac{3}{7}$ & $(\frac{3}{7},\frac{5}{9})$ & $\frac{5}{9}$ \\ \hline
			{$\mathrm{Sing}(S)$ }  & $\boldsymbol{A}_4$ & $\boldsymbol{A}_4$ & $\boldsymbol{A}_5$, $\boldsymbol{D}_4$ & $\boldsymbol{A}_5$, $\boldsymbol{D}_4$ \\
			\hline
			{$\mathrm{Sing}(D)$ }& isolated or cuspidal at $\boldsymbol{A}_1\in S$ & tacnodal or normal crossings at $\boldsymbol{A}_1\in S$  & tacnodal or normal crossings at $\boldsymbol{A}_1\in S$& cuspidal  or normal crossings at $\boldsymbol{A}_1\in S$\\
			\hline\hline
			$t$ & $(\frac{5}{9},\frac{9}{13})$ & $\frac{9}{13}$& $(\frac{9}{13},1)$ & \\
			\hline
			{$\mathrm{Sing}(S)$ }  & $\boldsymbol{A}_5$, $\boldsymbol{D}_5$ & $\boldsymbol{A}_5$, $\boldsymbol{D}_5$  & ${\boldsymbol{E}}_6$ & \\
			\hline
			{$\mathrm{Sing}(D)$ }& cuspidal  or normal crossings at $\boldsymbol{A}_1\in S$& normal crossings on smooth or $\boldsymbol{A}_1\in S$& normal crossings on smooth or $\boldsymbol{A}_1\in S$&  \\
			\hline
			
	\end{tabular} \end{center}
	\caption{Worst possible singularities in a $t$-stable pair $(S,D)$ for each $t\in(0,1)$.}
\end{table}

See Table \ref{tab:cubic-curves} to reinterpret $D$ in the language of ADE singularities. Our first classification result describes the stable orbits of $M(t)$ in terms of their singularities:
\begin{theorem}\label{theorem:stableGIT} Consider a pair $(S,D)$ formed by a cubic surface $S$ and a hyperplane section $D\in |-K_S|$.
	\begin{enumerate}[(i)]
		\item Let $t\in (0, \frac{1}{5})$. The pair $(S,D)$ is $t$-stable if and only if $S$ has finitely many singularities at worst of type $\boldsymbol{A}_2$ and if $P\in D$ is a surface singularity, then $P$ is at worst an $\boldsymbol{A}_1$ singularity of $S$. In particular $D$ may be non-reduced.
		\item Let $t=\frac{1}{5}$. The pair $(S,D)$ is $t$-stable if and only if $S$ has finitely many singularities at worst of type $\boldsymbol{A}_2$, $D$ is reduced, and if $P\in D$ is a surface singularity, then $P$ is at worst an $\boldsymbol{A}_1$ singularity of $S$.
		\item Let $t\in (\frac{1}{5}, \frac{1}{3})$. The pair $(S,D)$ is $t$-stable if and only if $S$ has finitely many singularities at worst of type $\boldsymbol{A}_3$, $D$ is reduced and if $P\in D$ is a surface singularity, then $P$ is at worst an $\boldsymbol{A}_1$ singularity of $S$.
		\item Let $t=\frac{1}{3}$. The pair $(S,D)$ is $t$-stable if and only if $S$ has finitely many singularities at worst of type $\boldsymbol{A}_3$, $D$ is reduced and if $P\in D$ is a surface singularity, then $P$ is at worst an $\boldsymbol{A}_1$ singularity of $S$ and $D$ has at worst a cuspidal singularity at $P$.
		\item Let $t\in (\frac{1}{3}, \frac{3}{7})$. The pair $(S,D)$ is $t$-stable if and only if $S$ has finitely many singularities at worst of type $\boldsymbol{A}_4$, $D$ is reduced and if $P\in D$ is a surface singularity, then $P$ is at worst an $\boldsymbol{A}_1$ singularity of $S$ and $D$ has at worst a normal crossing singularity at $P$ as a plane cubic curve.
		\item Let $t=\frac{3}{7}$. The pair $(S,D)$ is $t$-stable if and only if $S$ has finitely many singularities at worst of type $\boldsymbol{A}_4$, $D$ has at worst a tacnodal singularity and if $P\in D$ is a surface singularity, then $P$ is at worst an $\boldsymbol{A}_1$ singularity of $S$ and $D$ has at worst a normal crossing singularity at $P$ as a plane cubic curve.
		\item Let $t\in (\frac{3}{7}, \frac{5}{9})$. The pair $(S,D)$ is $t$-stable if and only if $S$ has finitely many singularities at worst of type $\boldsymbol{A}_5$ or $\boldsymbol{D}_4$, $D$ has at worst a tacnodal singularity and if $P\in D$ is a surface singularity, then $P$ is at worst an $\boldsymbol{A}_1$ singularity of $S$ and $D$ has at worst a normal crossing singularity at $P$ as a plane cubic curve.
		\item Let $t=\frac{5}{9}$. The pair $(S,D)$ is $t$-stable if and only if $S$ has finitely many singularities at worst of type $\boldsymbol{A}_5$ or $\boldsymbol{D}_4$, $D$ has at worst an $\boldsymbol{A}_2$ singularity and if $P\in D$ is a surface singularity, then $P$ is at worst an $\boldsymbol{A}_1$ singularity of $S$ and $D$ has at worst a normal crossing singularity at $P$ as a plane cubic curve.
		\item Let $t\in (\frac{5}{9}, \frac{9}{13})$. The pair $(S,D)$ is $t$-stable if and only if $S$ has finitely many singularities at worst of type $\boldsymbol{A}_5$ or $\boldsymbol{D}_5$,  $D$ has at worst a cuspidal singularity and if $P\in D$ is a surface singularity, then $P$ is at worst an $\boldsymbol{A}_1$ singularity of $S$ and $D$ has at worst a normal crossing singularity at $P$ as a plane cubic curve.
		\item Let $t=\frac{9}{13}$. The pair $(S,D)$ is $t$-stable if and only if $S$ has finitely many singularities at worst of type $\boldsymbol{A}_5$ or $\boldsymbol{D}_5$,  $D$ has at worst normal crossing singularities  as a plane cubic curve and if $P\in D$ is a surface singularity, then $P$ is at worst an $\boldsymbol{A}_1$ singularity of $S$.
		\item Let $t\in (\frac{9}{13}, 1)$. The pair $(S,D)$ is $t$-stable if and only if $S$ has finitely many ADE singularities, $D$ has at worst normal crossing singularities  as a plane cubic curve and if $P\in D$ is a surface singularity, then $P$ is at worst an $\boldsymbol{A}_1$ singularity of $S$.
	\end{enumerate}
\end{theorem}

The next theorem gives a full of classification of the pairs $(S,D)$ associated to each of the unique closed orbits in $\overline{M(t)}\setminus M(t)$ for each $t\in (0,1)$. Normal cubic surfaces with a $\mathbb C^*$-action 
have been classified \cite[{Table $3$}]{du2000hypersurfaces}. They play a central role in our classification,
as they are all realized as part of some strictly semistable log pair of some wall.

Figure \ref{fig:GITSB} gives sketches of each of these pairs and Table \ref{tab:VGITcubics} summarises these orbits.
Recall that an \emph{Eckardt point} of a cubic surface $S$ is a point where three coplanar lines of $S$ intersect.
\begin{figure}[!h]\label{fig:GITSB}
	\begin{center}
		\includegraphics[scale=.5]{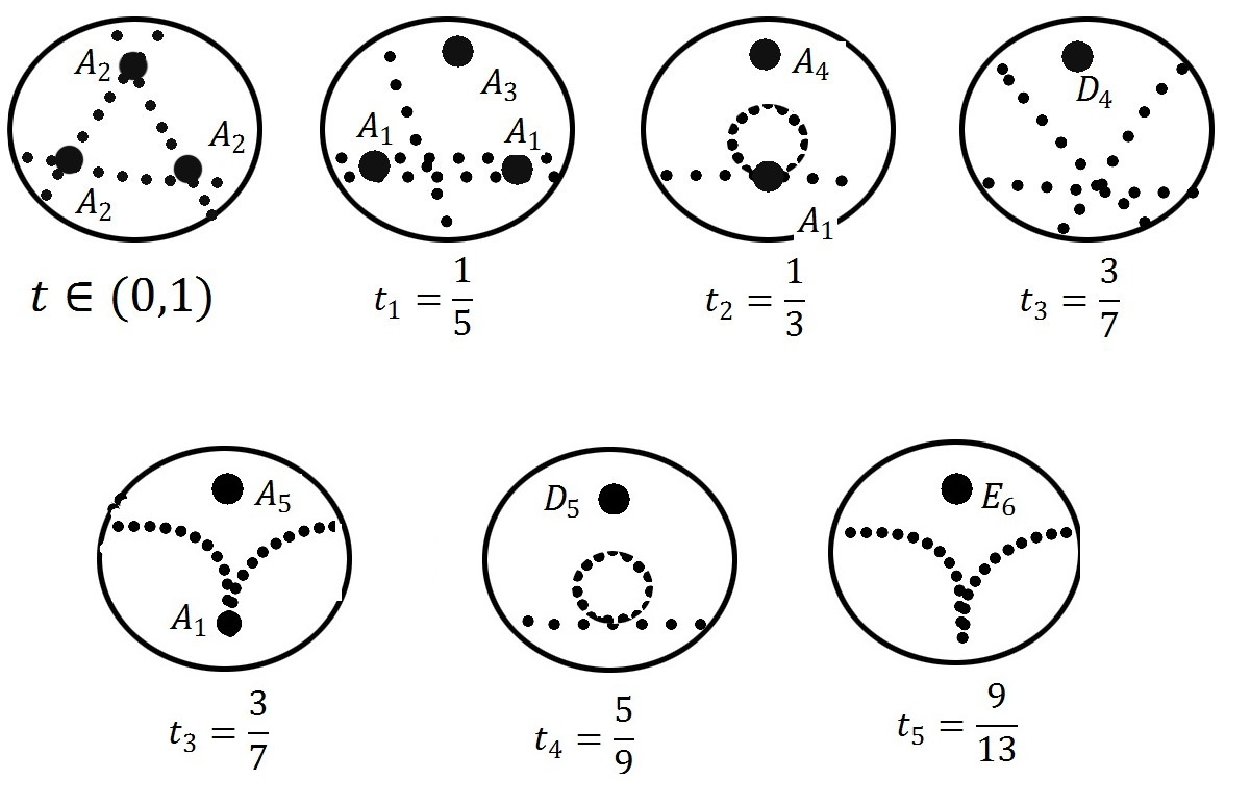}
		\caption{Pairs in $\overline{M(t)}\setminus M(t)$ for each $t\in(0,1)$. The dotted lines represent the divisor $D$. The bold points are singularities of the surface.}
		
	\end{center}
\end{figure}
\begin{theorem}\label{theorem:GITb}
	Let $t\in (0,1)$. If $t\neq t_i$, then $\overline {M(t)}$ is the compactification of the stable loci $M(t)$ by the closed $\SL(4,\mathbb C)$-orbit in $\overline {M(t)}\setminus M(t)$ represented by the pair $(S_0,D_0)$, where $S_0$ is the unique $\mathbb C^*$-invariant cubic surface with three $\boldsymbol{A}_2$ singularities and $D_0$ is the union of the unique three lines in $S_0$, each  of them passing through two of those singularities.
	
	If $t=t_i$, $i=1,2,4,5$, then $\overline {M(t_i)}$ is the compactification of the stable loci $M(t_i)$ by the two closed $\SL(4, \mathbb C)$-orbits in $\overline {M(t_i)}\setminus M(t_i)$ represented by the uniquely defined pair $(S_0,D_0)$ described above and the $\mathbb C^*$-invariant pair $(S_i,D_i)$ uniquely defined as follows:
	\begin{enumerate}[(i)]
		\item the cubic surface $S_1$ with an $\boldsymbol{A}_3$ singularity and two $\boldsymbol{A}_1$ singularities and the divisor $D_1=2L+L'\in |-K_S|$, where $L$ and $L'$ are lines such that $L$ is the line containing both $\boldsymbol{A}_1$ singularities and $L'$ is the only line in $S$ not containing any singularities;
		\item the cubic surface $S_2$ with an $\boldsymbol{A}_4$ singularity and an $\boldsymbol{A}_1$ singularity and the divisor $D_2\in |-K_S|$, which is a tacnodal curve singular at the $\boldsymbol{A}_1$ singularity of $S$;
		\item the cubic surface $S_4$ with a $\boldsymbol{D}_5$ singularity and the divisor $D_4\in |-K_S|$, which is a tacnodal curve that does not contain the surface singularity;
		\item the cubic surface $S_5$ with an ${\boldsymbol{E}}_6$ singularity and the cuspidal rational curve $D_5\in |-K_S|$ that does not contain the surface singularity. 
	\end{enumerate}
	
	The space $\overline {M(t_3)}$ is the compactification of the stable loci $M(t_3)$ by the three closed $\SL(4, \mathbb C)$-orbits in $\overline {M(t_3)}\setminus M(t_3)$ represented by the $\mathbb C^*$-invariant pairs uniquely defined as follows:
	\begin{enumerate}[(i)]
		\item the pair $(S_0,D_0)$ described above;
		\item the pair $(S_3,D_3)$, where $S_3$ is the cubic surface with a $\boldsymbol{D}_4$ singularity and and Eckardt point and $D_3$ consists of the unique three coplanar lines intersecting at the Eckardt point;
		\item the pair $(S_3',D_3')$, where $S_3'$ is the cubic surface with an $\boldsymbol{A}_5$ and an $\boldsymbol{A}_1$ singularity and the divisor $D_3'$, which is an irreducible curve with a cuspidal point at the $\boldsymbol{A}_1$ singularity of $S_3'$.
	\end{enumerate}
\end{theorem}

\begin{table}[!b]\label{tab:VGITboundary}
	\begin{center}
		\begin{tabular}{ |P{1cm}||P{2.1cm}|P{1.8cm}|P{1.8cm}|P{1.8cm}|}\hline
			$t$ & $(0,1)$ & \multicolumn{2}{|c|}{ $\frac{1}{5}$}&  $\frac{1}{3}$ \\  
			\hline
			{$\mathrm{Sing}(S)$ } & 3$\boldsymbol{A}_2$ & \multicolumn{2}{|c|}{$\boldsymbol{A}_3+2\boldsymbol{A}_1$}& $\boldsymbol{A}_4+\boldsymbol{A}_1$\\
			\hline
			{$D$} & unique three lines in $S$ & \multicolumn{2}{|p{3.9cm}|}{double line containing $2\boldsymbol{A}_1$ and unique line not containing surface singularities} & Tacnodal curve at $\boldsymbol{A}_1$\\
			\hline\hline
			$t$ & $\frac{3}{7}$ & $\frac{3}{7}$ & $\frac{5}{9}$ & $\frac{9}{13}$ \\ \hline
			{$\mathrm{Sing}(S)$ }  &$\boldsymbol{D}_4$, $S$ has an Eckardt point $p$ &  $\boldsymbol{A}_5+\boldsymbol{A}_1$ &  $p=\boldsymbol{D}_5$  & $p=\boldsymbol{E}_6$\\
			\hline
			{$D$ }& Unique three coplanar lines through $p$  &  $\mathbb C^*$-invariant cuspidal curve at $\boldsymbol{A}_1$, & $\mathbb C^*$-invariant tacnodal curve, $p\not\in D$ & $\mathbb C^*$-invariant cuspidal curve, $p\not\in D$\\
			\hline
			
	\end{tabular} \end{center}
	\caption{Strictly semistable pairs with closed orbits appearing in each $t\in(0,1)$.}
\end{table}

The theory of variations of GIT quotients used to construct these quotients can be used to understand the birational maps among them. In particular, for $\varepsilon>0$ sufficiently small, we have morphisms $\overline {M(\varepsilon)}\rightarrow \overline {M(0)}$ and $\overline{M(1-\varepsilon)}\rightarrow \overline{M(1)}$. 

 By Pinkham's theory on deformation of singularities with $\mathbb C^*$-action, the deformations of negative weight  can be globalized and interpreted as a moduli space of pairs (see \cite[Theorem 2.9]{pinkham1978deformations}). In particular,  the fiber of the map $M(1-\varepsilon) \to {M(1)}$ over a point representing a smooth curve
 with trivial stabilizer
  is isomorphic to the deformation of the $\tilde E_6$ singularity in negative direction modulo the natural action of $\mathbb C^*$ (c.f. \cite[Section 2.4]{laza2009deformations} for an analogue situation with the $N_{16}$ singularity).
Such deformations of $\tilde E_6$ were determined by Looijenga \cite[Theorem 3.4]{Looijenga-Root-Systems}. To make this explicit, let $E$ be a smooth elliptic curve and $p_E \in {M(1)\subset \overline{M(1)}}$ be the point representing it.  The fiber over $p_E$ of $M(1-\epsilon) \to {M(1)}$ is isomorphic to  $(E \otimes E_6)/W(E_6) \cong \mathbb P(1,g_1,g_2,g_3,g_4,g_5,g_6)$ where $g_i$ are the coefficients of the highest root of $E_6$ with respect to a set of simple roots, i.e the fiber is isomorphic to $\mathbb P(1,1,1,2,2,2,3)$.

\section{GIT set-up and computational methods}\label{sec:GIT-setup} 
In this section, we briefly describe the GIT setting for constructing  our compact moduli spaces.  
We refer the reader to \cite{VGITour}, where the problem is thoroughly discussed and solved for pairs formed by a hyperplane and a hypersurface of $\mathbb P^{n+1}$ of a fixed degree. 
Our GIT quotients are given by
\begin{equation*}
\overline{M} \left( \frac{b}{a} \right)
:=
\left( 
\mathbb{P} (H^0(\mathbb{P}^{3}, \oo_{\mathbb{P}^3}(3))) \times \mathbb{P}(H^0(\mathbb{P}^3, \oo_{\mathbb{P}^3}(1)))
\right)^{ss} 
\Big/\!\!\!\!\Big/_{\oo(a,b)} \mathrm{SL}(4, \mathbb{C}),
\end{equation*}
and they depend only of one parameter $t:=\frac{b}{a}\in\mathbb Q_{\geq 0}$.    The use of GIT requires three initial combinatorial steps which are computed with the algorithm described in \cite{VGITour} and implemented in \cite{gallardo-jmg-code}.
 The first step is to find a set of \emph{candidate GIT walls} which includes all GIT walls (see \cite[Theorem 1.1]{VGITour}). Some of these walls may be redundant and they are removed by comparing if there is any geometric change to the $t$-(semi)stable pairs $(S,D)$ for $t=t_i\pm \epsilon$ for $0<\epsilon\ll 1$. The set of candidate GIT walls is precisely the one in Lemma \ref{lemma:walls} and once Theorem \ref{theorem:GITb} is proven this proves Lemma \ref{lemma:walls}.

The second step (see \cite[Lemma 3.2]{VGITour}) is to find the finite set $S_{2,3}$ of one-parameter subgroups that determine the $t$-stability of all pairs $(S,D)$ for all $t$. For convenience,  given a one-parameter subgroup $\lambda=\mathrm{Diag}(r_0, \ldots, r_3)$, we define its dual one as  $\overline{\lambda}=\mathrm{Diag}(-r_3, \ldots, -r_0)$.

\begin{lemma}
\label{lemma:critical1PS}
The elements $S_{2,3}$ are $\lambda_k$ and $\overline{\lambda}_k$ where $\lambda_k$ is one of the following:
\begin{align*}
\lambda_{1} = \mathrm{Diag}(1, 0, 0, -1)  
& &\lambda_{2} =  \mathrm{Diag}(2, 0, -1, -1)
& &\lambda_{3} =\mathrm{Diag}(5, 1, -3, -3)\\
\lambda_{4}=\mathrm{Diag}(13, 1, -3, -11) 
& &\lambda_{5} =\mathrm{Diag}(3, 1, -1, -3) 
& &\lambda_{6} = \mathrm{Diag}(9, 1, -3, -7) \\ 
\lambda_{7}=\mathrm{Diag}(5,5,-3,-7) 
& &\lambda_{8}=\mathrm{Diag}(1, 1, 1, -3)   
& &\lambda_{9}= \mathrm{Diag}(5, 1, 1, -7)\\
\lambda_{10}=\mathrm{Diag}(1, 1, -1, -1) 
&
\end{align*}
\end{lemma}
Let $\Xi_k$ be the set of all  monomials in four variables of degree $k$. Let $g\in \SL(4,\mathbb C)$. Suppose $g\cdot S$ is given by the vanishing locus of a homogeneous polynomial $F$ of degree $3$ and $g\cdot D$ is given by the vanishing locus of $F$ and a homogeneous polynomial $l$ of degree $1$.  We say that $F$ and $l$ are \emph{associated} to the pair $(g\cdot S,g\cdot D)$ and to the corresponding pair of sets of monomials. Let $\lambda=\mathrm{Diag}(r_0,\ldots, r_3)$. Denote by $\mathcal S\subseteq \Xi_3$ and $\mathcal D\subseteq \Xi_1$ the monomials with non-zero coefficients in $F$ and $l$, respectively. There is a natural pairing $\langle v, \lambda\rangle\in \mathbb Z$ for any $v\in \Xi_k$, namely $\langle x^{i_0}_0\cdots x^{i_3}_3, \mathrm{Diag}(r_0,\ldots, r_3)\rangle=\sum i_jr_j$. We define
$$\mu_t(g \cdot S, g \cdot D,\lambda)\coloneqq\min_{v\in \mathcal S}\langle v, \lambda \rangle+t\min_{x_i\in \mathcal D}\langle x_i, \lambda\rangle.$$

\begin{lemma}[{Hilbert-Mumford Criterion, see \cite[Lemma 3.2]{VGITour}}]
A pair $(S,D)$, where $D=S\cap H$, is not $t$-stable if and only if there is $g \in \SL_n$ satisfying
$$\mu_t(S,D) = \max_{\substack{\lambda\in S_{2,3}}}\{\mu_t(g \cdot S, g \cdot D,\lambda)\} \geqslant 0.$$
\end{lemma}
Given $t\in(0,1)$, and $\lambda\in S_{2,3}$ and $i\in \{0,\ldots,3\}$, the next step is to find the pairs of sets $N^{\oplus}_t(\lambda,x_i) 
\coloneqq\left(V^\oplus_t(\lambda, x_i), B^\oplus(x_i)\right)$ defined as:
\begin{align*}
V^\oplus_t(\lambda, x_i) 
&=
 \{v \in \Xi_d \ | \ \langle v, \lambda\rangle + t\langle x_i,\lambda\rangle> 0\}, 
& &
B^\oplus(x_i)=\{x_k \in \Xi_1 \ | \ k\leqslant i\},   \notag
\end{align*} 
which are maximal with respect to the containment order. Since by  \cite[Lemma 3.2]{VGITour}, we only need to consider the one-parameter subgroups in Lemma \ref{lemma:critical1PS}, which is a finite computation. Hence, they can be computed using computer software \cite{gallardo-jmg-code}. A more detailed algorithm can be found in \cite{VGITour}.
\begin{theorem}[{\cite[Theorem  1.4]{VGITour}}]
\label{theorem:unstMain}
Let $t\in (0,1)$. A pair $(S,S \cap H)$ is not $t$-stable if and only if there exists $g \in \SL(4, \mathbb{C})$ such that the set of monomials associated to  $(g\cdot S, g \cdot H)$ is contained in a pair of sets $N^\oplus_t(\lambda, x_i)$.
\end{theorem}
Given $N^\oplus_t(\lambda, x)$, define $N^{0}_t(\lambda,x_i)  \coloneqq\left(V^0_t(\lambda, x_i), B^0(x_i)\right)$  (see \cite[Proposition 5.3]{VGITour}) where
$$
V^0_t(\lambda, x_i)\times B^0(x_i)=\{(v,m)\in V^\oplus_t(\lambda, x_i)\times B^\oplus(x_i) \ | \ \langle v, \lambda \rangle + t\langle m, \lambda \rangle=0\}.
$$
\begin{theorem}[{\cite[Theorem 1.6]{VGITour}}]
\label{theorem:closedOrbit}
Let $t\in (0,1)$. If a pair $(S,S\cap H)$ belongs to a closed strictly $t$-semistable orbit, then there exist $g\in \SL(4,\mathbb C)$, $\lambda\in S_{2,3}$ and $x_i$ such that the set of monomials associated to $(g\cdot S, g\cdot D)$ corresponds to those in a pair of sets $N^0_t(\lambda,x_i)$.
\end{theorem}

\section{Preliminaries in singularity theory}\label{sec:sing}
We recall the admissible singularities in normal cubic surfaces.
\begin{proposition}[{\cite{classificationcubics}}]
\label{proposition:singcub}
Let $X$ be an irreducible and reduced cubic surface and $p \in X$ be an isolated singular point. Then, the singularity at $p$ is either a Du val singularity (of type $\boldsymbol{A}_k$, $\boldsymbol{D}_k$ with  $ k \leq 5$ or ${\boldsymbol{E}}_6$), or a cone over a smooth elliptic curve (i.e. a simple elliptic singularity  of type $\widetilde{\boldsymbol{E}}_6$). 
\end{proposition}

\begin{definition}[{\cite[p.88]{arnoldinvetiones}}]
\label{def:dgd}
A class of singularities $T_2$ is adjacent to a class $T_1$, and one writes 
$T_1 \gets T_2$ if every germ of $f \in T_2$  can be locally deformed  into a germ in  $T_1$ by an arbitrary small deformation. We say that  the singularity $T_2$ is \emph{worse} than $T_1$; or that $T_2$ is a \emph{degeneration} of $T_1$. 
\end{definition}
The  degenerations of the isolated singularities that appear in a cubic surface (or in their anticanonical divisors, which are plane cubic curves) are  described in Figure \ref{fig:adj} (for details see \cite[p. 88]{arnoldinvetiones} and \cite[\S 13]{ArnoldLong}). 
\begin{figure}[h!]
\centerline{
\xymatrix{
\boldsymbol{A}_1 & \boldsymbol{A}_2 \ar[l]  & \boldsymbol{A}_3 \ar[l] & \boldsymbol{A}_4 \ar[l]  & \boldsymbol{A}_5  \ar[l]
\\
       &                      &  \boldsymbol{D}_4   \ar[u]      & \boldsymbol{D}_5  \ar[l] \ar[u] &  {\boldsymbol{E}}_6 \ar[l] \ar[u] & \widetilde{\boldsymbol{E}}_6 \ar[l]
}	
}
\caption{Degeneration of germs of isolated singularities appearing in cubic surfaces.}
\label{fig:adj}
\end{figure}
The above theory considers only local deformations of singularities. When we study degenerations in the GIT quotient we are interested in global deformations. 
\begin{lemma}[{\cite[Theorem  1]{shustin1999versal}, c.f. \cite{Hacking-Prokhorov}}]\label{lemma:versal}
Let $V(T_1, \ldots T_r )$ be the set of cubic hypersurfaces in $\mathbb P^n$ for $n\leqslant 3$ with $r$ isolated singular points of types $T_1, \ldots T_r$. The germ of the linear system $|\mathcal{O}_{\mathbb{P}^3}(3)|$ at any $X\in V( T_1, \ldots T_r )$ is a joint versal deformation of all singular points of $X$ if $\sum_{i=1}^{r}\mu(T_i) \leq 9$ where $\mu(T_i)$ is the Milnor number  of $T_i$. 
\end{lemma}
Recall that $\mu(\boldsymbol{A}_k)=k$, $\mu(\boldsymbol{D}_k)=k$ and $\mu({\boldsymbol{E}}_6)=6$. By checking carefully how these singularities appear together in each cubic surface (see \cite[p. 255]{classificationcubics}) we conclude that $\sum^{r}_{i=1}\mu(T_i)\leqslant 6$ for all cubic surfaces with ADE singularities. Furthermore, by looking at Table \ref{tab:cubic-curves}, we see that $\sum^{r}_{i=1}\mu(T_i)\leqslant 4$ for any plane cubic curve with isolated singularities . Hence, Lemma \ref{lemma:versal} implies that for cubic plane curves and cubic surfaces, any local deformation of isolated singularities is induced by a global deformation.

\begin{definition}[{\cite{classificationcubics}}]
\label{def:ADEweights}
A polynomial $F$ in $n+1$ variables is \emph{semi-quasi-hom\-o\-gen\-eous (SQH)} with respect to the weights
$(w_1,w_2,\ldots, w_{n})$ if all the  monomials of $F$ have weight larger or equal than $1$ and those monomials of weight $1$ define a function with an isolated singularity. In particular, the weights associated to the ADE singularities $\boldsymbol{A}_k$, $\boldsymbol{D}_k$ and ${\boldsymbol{E}}_6$ are
\begin{align*}
\left( \frac{1}{2},\ldots,\frac{1}{2},\frac{1}{k+1} \right),
& & \left( \frac{1}{2}, \ldots, \frac{1}{2},\frac{(k-2)}{2(k-1)}, \frac{1}{k-1} \right),
& &
\left( \frac{1}{2}, \ldots,\frac{1}{2}, \frac{1}{3}, \frac{1}{4} \right),
\end{align*}
respectively. Furthermore, the weight of $\widetilde {\boldsymbol{E}}_6$ is $\left(\frac{1}{2},\ldots,\frac{1}{2}, \frac{1}{3}, \frac{1}{3}, \frac{1}{3} \right)$. These weights are uniquely associated to their respective singularity.
\end{definition}

\begin{lemma}[{\cite[p. 246]{classificationcubics}}]
\label{lemma:recog}
If  $F(x_0,x_1,x_2)$ is SQH with respect to one of the sets of weights in Definition \ref{def:ADEweights} we can, by a locally analytic change of coordinates, reduce the terms of weight $1$ to the normal forms for $\boldsymbol{A}_k$, $\boldsymbol{D}_k$, ${\boldsymbol{E}}_6$, which are locally analytically isomorphic to the following surface singularities:
\begin{align*}
&\boldsymbol{A}_k\colon \ x_1^{k+1}+ x_2^2+x_3^2\ (k\geq 1), &\boldsymbol{D}_k\colon \ x_1^{k-1}+x_1x_2^2 + x_3^2\ (k\geq 4),\qquad\qquad\qquad\\
&{\boldsymbol{E}}_6\colon \ x_1^{3}+ x_2^4+x_3^2, &\widetilde{\boldsymbol{E}}_6\colon \ x_1^3+x_2^3+x_3^3+3\lambda x_1x_2x_3,\quad \lambda^3\neq -1. 
\end{align*}
and the resulting function will remain SQH.
\end{lemma}

Reduced plane cubic curves are completely characterized according to the number and type of their ADE singularities (see Table \ref{tab:cubic-curves}).

\begin{table}[!b]%
\begin{center}
\begin{tabular}{|p{4.2cm}|c||p{4.4cm}|c|}
\hline
Non-singular & - & Cuspidal cubic & $\boldsymbol{A}_2$
\\\hline  
Nodal cubic & $\boldsymbol{A}_1$ &  Three lines intersecting normally& $3\boldsymbol{A}_1$
\\ \hline
Line and conic intersecting  normally & $2\boldsymbol{A}_1$  & 
Three lines intersecting at a point & $\boldsymbol{D}_4$
\\\hline
Line and conic tangent at a point & $\boldsymbol{A}_3$ & &
\\ \hline
\end{tabular}
\caption{
Classification of plane cubic curves with isolated singularities.
}

\label{tab:cubic-curves}
\end{center}
\end{table}

\section{ Geometric characterization of pairs}\label{sec:1ps}
In this section we relate the classifications of pairs in terms of singularity theory and the equations defining them. We have divided our lemmas in four  groups: classification of singular cubic surfaces, classification of pairs $(S,D)$ with singular boundary $D$, classification of pairs $(S,D)$ where $S$ is singular at a point $P\in D$ and classification of pairs $(S,D)$ invariant under a $\mathbb C^*$-action. We will denote homogenous polynomials of degree $d$ in $n+1$ variables as $f_d(x_0,\ldots,x_{n}), g_d$, etc. Recall that pairs $(S, D)$ and $(S',D')$ are projectively equivalent if and only if they are conjugate to each others by elements of $\mathrm{Aut}(\mathbb P^3)$. 

\begin{lemma}[{\cite[Lemma 3]{classificationcubics}}]
\label{lemma:lem3}
Let $F=x_0x_1x_3+f_3(x_0,x_1,x_2)$, $P=(0,0,0,1)$, $Q=(0,0,1,0)$, $H=\{x_3=0\}\cong \mathbb P^2_{(x_0,x_1,x_2)}$ and $H_i=\{x_i=x_3=0\}\subset H$ for $i=0,1$.
\begin{enumerate}
\item The singularities of $\{F=0\}$ other than that at $P$ correspond to the intersection of $C=\{x_0x_1=0\}\subset H$ and $C'=\{f_3=0\}$ at points $R$ other than $Q$. Indeed, if $\mult_{R}(C\cdot C')=k$, then $R$ is an $\boldsymbol{A}_{k-1}$ singularity.

\item If $f_3(0,0,1) \neq 0$, then $P$ is an $\boldsymbol{A}_2$ singularity. Let $k_i=\mult_{Q}(H_i\cdot C')$. If both $k_0$ and $k_1$ are both at least $2$, then $\{F=0\}$ has non-isolated singularities. Otherwise $P$ is an $\boldsymbol{A}_{k_0+k_1+1}$ singularity for $\{ k_0,k_1 \}=\{1,1 \}$, $\{1,2\}, \{1,3\}$. \end{enumerate}
\end{lemma}

\begin{lemma}
 \label{lemma:A1D}
A pair $(S,D)$ is such that $S$ has an $\boldsymbol{A}_2$ singularity at a point $P\in D$ or a degeneration of one if and only if $P$ is conjugate to $(0,0,0,1)$ and simultaneously $(S,D)$ is projectively equivalent to the pair defined by equations
\begin{align*}
x_3f_2(x_0,x_1)+f_3(x_0,x_1,x_2)=0,
\qquad{}
l_1(x_0,x_1,x_2)=0.
\end{align*}

\end{lemma}
\begin{proof}
Without loss of generality, we may assume $P=(0,0,0,1)$. By Lemma \ref{lemma:lem3}, $S$ has (a degeneration of) an $\boldsymbol{A}_2$ singularity at $P$ if and only if it is given by the equation $x_0x_1x_3+f_3(x_0,x_1,x_2)=0$. Any quadric $f_2(x_0,x_1)$ can be transformed to $x_0x_1$ or to a degeneration of $x_0x_1$ (e.g. $x_0^2$) by a change of coordinates preserving $x_2$ and $x_3$. The lemma follows because a hyperplane section $D$ contains $P$ if and only if $D$ is given by a linear form $l_1(x_0,x_1,x_2)$.
\end{proof}

\begin{lemma}
\label{lemma:A3}
A surface $S$ has an $\boldsymbol{A}_3$ singularity or a degeneration of one if and only if it is projectively equivalent to:
\begin{equation*}
\{x_3f_2(x_0,x_1)+x_2^2f_1(x_0,x_1)+x_2g_2(x_0,x_1)+g_3(x_0,x_1)=0\}.
\end{equation*}
\end{lemma}
\begin{proof}
By Lemma \ref{lemma:lem3}, we may assume $S=\{x_0x_1x_3+f_3(x_0,x_1,x_2)=0\}$ and $P=(0,0,0,1)$. Moreover, the singularity is of type $\boldsymbol{A}_k$ with $k\geqslant 3$ if and only if $f_3(0,0,1)=0$. Therefore $f_3(x_0,x_1,x_2)=x_2^2f_1(x_0,x_1)+x_2g_2(x_0,x_1)+g_3(x_0,x_1)$. 

\end{proof}

\begin{lemma}\label{lemma:A4}
A surface $S$ has an $\boldsymbol{A}_4$ singularity or a degeneration of one if and only if it is projectively equivalent to $\{x_3x_0l_1(x_0,x_1)+x_0x_2^2+x_2g_2(x_0,x_1)+g_3(x_0,x_1)=0\}$.
\end{lemma}
\begin{proof}
By Lemma \ref{lemma:lem3}, the surface $S$ is defined by the equation
$$x_0x_1x_3+f_3(x_0,x_1,x_2)=0,$$
where $f_3(x_0x_1x_2)=x_2^2f_1(x_0,x_1)+x_2g_2(x_0,x_1)+g_3(x_0,x_1)$, $k_0=\mult_{Q}(H_0\cdot C')\geqslant 2$ and $k_1=\mult_{Q}(H_1\cdot C')\geqslant 1$ if and only if $P$ is (a degeneration of) an $\boldsymbol{A}_4$ singularity, where $C'$ is the curve given in Lemma \ref{lemma:lem3}. Notice that
$$k_i=\mult_Q(H_i\cdot C')=\dim_{\mathbb C}\left(\frac{\mathbb C[x_0,x_1]}{\langle x_i,f_1+g_2+g_3\rangle}\right).$$
Therefore $k_0\geqslant 2$ if and only if $f_1(0,1)=0$. Hence, $f_1=x_0$. The lemma follows from noticing that $x_0x_1x_3$ is projectively equivalent to $x_0x_3l_1(x_0,x_1)$ by an element of $\mathrm{Aut}(\mathbb P^3)$ fixing $x_0, x_2, x_3$.
\end{proof}

The proof of the next lemma is similar to the proof of Lemma \ref{lemma:lem3}, so we omit it.
\begin{lemma}\label{lemma:A5}
A surface $S$ has an $\boldsymbol{A}_5$  singularity or a degeneration of one if and only if it is projectively equivalent to
$$\{x_3x_0l_1(x_0,x_1)+x_0x_2f_1(x_0,x_1,x_2)+f_3(x_0,x_1)=0\}.$$
\end{lemma}

In Figure \ref{fig:adj} we see that the only non-trivial degenerations of a $\boldsymbol{D}_4$ singularity in a cubic surface which are not a $\tilde{{\boldsymbol{E}}_6}$ singularity are $\boldsymbol{D}_5$ and ${\boldsymbol{E}}_6$ singularities. Hence the next lemma follows at once from \cite[Case C]{classificationcubics}.
\begin{lemma}
\label{lemma:D4}
A surface $S$ has a $\boldsymbol{D}_4$ singularity or a degeneration of one if and only if it is projectively equivalent to
$$\{x_3x_0^2+f_3(x_0,x_1,x_2)=0\}.$$
\end{lemma}

\begin{lemma}\label{lemma:D5}
A surface $S$ has a $\boldsymbol{D}_5$ singularity or a degeneration of one if and only if it is projectively equivalent to
$$\{f_3(x_0,x_1)+x_2g_2(x_0,x_1)+x_0x_2^2+x_0^2x_3=0\}.$$
\end{lemma}
\begin{proof}
By Lemma \ref{lemma:D4} and Figure \ref{fig:adj}, we may assume that $S$ is given by $x_3x_0^2+f_3(x_0,x_1,x_2)$ since $\boldsymbol{D}_5$ is a degeneration of $\boldsymbol{D}_4$. Let $H=\{x_3=0\}$, $C=\{x_3=f_3(x_0,x_1,x_2)=0\}\subset H$ and $C'=\{x_3=x_0=0\}\subset H$. We can rewrite $f_3=x_2^2g_1(x_0,x_1)+x_2g_2(x_0,x_1)+g_3(x_0,x_1)$.
By \cite[Lemma 4]{classificationcubics}, the point $P=(0,0,0,1)$ is (a degeneration of) a $\boldsymbol{D}_5$ singularity if and only if $C\cap C'$ consist of at most two points. The equation of $S\cap H\subset H$ localized at $Q=(0,0,1,0)$ is $g_1(x_0,x_1)+g_2(x_0,x_1)+g_3(x_0,x_1)=0$, and $C\cap C'$ consists of at most two points if and only if
$$\dim_{\mathbb C}\left(\frac{\mathbb C[x_0,x_1]}{\langle x_0,g_1+g_2+g_3\rangle}\right)\geqslant 2.$$
The latter is equivalent to taking $g_1=ax_0$, which by rescaling $x_2$ gives the result.
\end{proof}

\begin{lemma}
\label{lemma:E6}
The unique cubic surface $S$ with a ${\boldsymbol{E}}_6$  singularity or a degeneration of one such surface is projectively equivalent to
$$\{x_3x_0^2+x_0x_2l_1(x_0,x_1,x_2)+f_3(x_0,x_1)=0\}.$$
\end{lemma}
\begin{proof}
Using the same notation as in Lemma \ref{lemma:D5} and following \cite[Lemma 4]{classificationcubics}, $S$ is defined by $x_3x_0^2+x_2^2g_1(x_0,x_1)+x_2g_2(x_0,x_1)+g_3(x_0,x_1)=0,$ and has (a degeneration of) an ${\boldsymbol{E}}_6$ singularity if and only if $$\dim_{\mathbb C}\left(\frac{\mathbb C[x_0,x_1]}{\langle x_0,g_1+g_2+g_3\rangle}\right)\geqslant 3.$$
The latter is equivalent to taking $g_1=x_0$ and $g_2=x_0l_1(x_0,x_1)$. 
\end{proof}

\begin{remark}[{see \cite[Case E]{classificationcubics}}]\label{remark:EE6}
A surface $S$ has an isolated $\widetilde{\boldsymbol{E}}_6$  singularity if and only if $S$ is the cone over a smooth plane cubic curve given by $f_3(x_0,x_1,x_2)=0$. 
\end{remark}

Consider a pair $(S,D)$ and a point $P\in D\subset S$. By choosing coordinates appropriately we can suppose that $P=(0,0,0,1)$ and $(S,D)=(\{F=0\},\{F=H=0\})$ for $F$ and $H$ given as
\begin{equation}
F=x_0f_2(x_0,\cdots,x_3)+x_3^2f_1(x_1,x_2)+x_3g_2(x_1,x_2)+f_3(x_1,x_2),\quad H=x_0.
\label{eq:pair-equations}
\end{equation} 

\begin{lemma}\label{lemma:cuspD}
A pair $(S,D)$ has $D$ with an $\boldsymbol{A}_2$ singularity at a point $P$ or a degeneration of one if and only if $(S,D)$ is projectively equivalent to the pair defined by equations:
\begin{align}
x_0f_2(x_0,x_1,x_2,x_3)+x_3x_1^2+f_3(x_1,x_2)=0,\qquad
x_0=0.
\end{align}
\end{lemma}
\begin{proof}
Without loss of generality we can suppose $(S,D)$ is given by \eqref{eq:pair-equations}. The equation of (a degeneration of) a plane cubic curve in $\{x_0=0\}$ with an $\boldsymbol{A}_2$ singularity at $P$ is given by $x_1^2x_3+f_3(x_1,x_2)=0$, where the curve has an $\boldsymbol{A}_2$ singularity at $P$ if and only if $x_2^3$ has a non-zero coefficient in $f_3$. Therefore $D$ is as in the statement if and only if in \eqref{eq:pair-equations} we take $f_1=0$ and $g_2=x_1^2$.
\end{proof}

\begin{lemma}\label{lemma:CL}
A pair $(S,D)$ has $D$ with an $\boldsymbol{A}_3$ singularity at $P$ or a degeneration of one if and only if $(S,D)$ is projectively equivalent to the pair defined by $x_0f_2(x_0,x_1,x_2,x_3)+x_1(x_2^2+x_1l_1(x_1,x_2,x_3))=0$ and $x_0=0$.
\end{lemma}
\begin{proof}
We may assume that the equations of $(S,D)$ are as in \eqref{eq:pair-equations} and $P=(0,0,0,1)$. By restricting to $\{x_0=0\}\cong \mathbb P^2$ and localizing at $P$, the equation for $D$ is $f_1(x_1,x_2)+g_2(x_1,x_2)+f_3(x_1,x_2)$
and by choosing coordinates appropriately we may assume that $L=\{x_1=0\}$ and $C=\{x_2^2+x_1l_1(x_1,x_2)=0\}$ are a line and a conic intersecting at $P$, where $l$ is a polynomial of degree $1$, not necessarily homogeneous. Therefore $D|_{x_0=0}$ has equation $x_1(x_2^2+x_1l_1(x_1,x_2,x_3))$ so $f_1\equiv 0$, $g_2\equiv ax_1^2$, $f_3=x_1x_2^2 + x_1l_1(x_1,x_2,0)$ and the result follows.
\end{proof}

By similar arguments, one can prove the next two results:
\begin{lemma}
\label{lemma:3LLL}
A pair $(S,D)$ has $D$ with a $\boldsymbol{D}_4$ singularity at $P$ or  a degeneration of one if and only if $(S,D)$ is projectively equivalent to the pair defined by equations $x_0f_2(x_0,x_1,x_2,x_3)+f_3(x_1,x_2)=0$ and $x_0=0$.

\end{lemma}

\begin{lemma}\label{lemma:2LL}
A pair $(S,D)$ has $D$ non-reduced if and only if  it is projectively equivalent  to the pair defined by equations:
$$x_0f_2(x_0,x_1,x_2,x_3)+x_1^2f_1(x_1,x_2,x_3)=0,\qquad x_0=0.$$
\end{lemma}
\begin{lemma}
\label{lemma:3L}
A pair $(S,D)$ has $D=L+C$ where $L$ is a line and $C$ is a conic such that $3L\in |-K_S|$ if and only if it is projectively equivalent  to the pair defined by equations:
\begin{equation*}
x_0f_2(x_0,x_1,x_2,x_3)+ax_1^3=0,\qquad l_1(x_0,x_1)=0,
\end{equation*}
where $L$ and $3L$ are projectively equivalent to $\{x_0=x_1=0\}$ and $=\{x_0=0\}|_S$, respectively. This surface has a point $Q\in L\subset \mathrm{Supp}(D)$ such that $S$ has a singularity at $Q$ that is not of type $\boldsymbol{A}_1$.
\end{lemma}
\begin{proof}
Suppose $(S,D)$ as in the statement. Without loss of generality, we may suppose that the equation of $S$ is as in \eqref{eq:pair-equations}, $D=\{x_0+b x_1=0\}$ and let $D':=\{x_0=0\}$. Clearly $L\subset \mathrm{Supp}(D')\cap \mathrm{Supp}(D)$ and $D=D'$ if and only if $b=0$. In this case, the equation of $D=D'$ in $\{x_0=0\}\cong \PP^2$ is given by
$x_3^2f_1(x_1,x_2)+x_3g_2(x_1,x_2)+f_3(x_1,x_2)=0$
and $3L\in |-K_S|$ if and only if $f_1=g_2\equiv 0$ and $f_3=ax_1^3$. If $b\neq 0$, then $x_1=-\frac{x_0}{b}$. Take $x_0=0$ in \eqref{eq:pair-equations}. The equation of $D'=\{x_0=0\}|_S$ is $x_3^2f_1+x_3g_2+f_3=0$ and $D'\equiv 3L$ if and only if $f_1=g_2=0$ and $f_3=x_1^3$. But then, the equation of $D$ in $\{x_0+bx_1=0\}$ is $x_1(bf_2+x_1^2)$ and $C=\{bf_2+x_1^2=x_0+bx_1=0\}$. It is a well known fact that the line $L$ contains a point $Q$ at which $S$ is singular and $Q$ is not of type $\boldsymbol{A}_1$ (see \cite[p. 227]{mukai}).
\end{proof}

\begin{lemma}\label{lemma:CuspS}
Given a pair $(S,D)$, $S$ is singular at a point $P\in D$ and $D$ is an $\boldsymbol{A}_2$ singularity at $P$ or a degeneration of one if and only if $(S,D)$ is projectively equivalent to the pair defined by equations:
\begin{align}
x_3x_0l_1(x_0,x_1,x_2)+x_3x_1^2+f_3(x_1,x_2)+x_0f_2(x_0,x_1,x_2)=0,
 & &
x_0=0.
\end{align}
\end{lemma}
\begin{proof}
Without loss of generality we can assume $P=(0,0,0,1)$. 
 Then, the equation of $S$ can be written as ({see \cite[Section 2, pp. 247--252]{classificationcubics}})
\begin{align*}
&\qquad x_3h_2(x_0,x_1,x_2)+h_3(x_0,x_1,x_2)=\\
&=a_0x_3x_1^2+x_0f_2(x_0,x_1,x_2)+f_3(x_1,x_2)+x_1x_3g_1(x_0,x_2)+x_3g_2(x_0,x_2).
\end{align*}
By comparing with the equation in Lemma \ref{lemma:cuspD}, $D$ has (a degeneration of) an $\boldsymbol{A}_2$ singularity at $P$ if and only if $g_1(x_0,x_2)=ax_0$ and $g_2(x_0,x_2)=bx_0^2+cx_0x_2$. The lemma follows.
\end{proof}

The proof of the next lemma is similar to that of Lemma \ref{lemma:CuspS}.
\begin{lemma}
\label{lemma:CLS}
Given a pair $(S,D)$, $S$ is singular at a point $P\in D$ and $D$ has an $\boldsymbol{A}_3$ singularity at $P$ or a degeneration of one if and only if $(S, D)$ is projectively equivalent  to the pair defined by equations:
$$ 
x_0^2l_1(x_0,x_1,x_2,x_3)+x_0f_2(x_1,x_2)  +x_0x_3g_1(x_1,x_2) +x_1^2h_1(x_1,x_2,x_3)
 +x_1x_2^2=0,
$$
$x_0 =0$.
\end{lemma}
\begin{lemma}
\label{lemma:Gm-action}
Let $(S,D)$ be a pair that is invariant under a non-trivial $\mathbb C^*$-action. Suppose the singularities of $S$ and $D$ are given as in the first and second entries in one of the rows of Table \ref{tab:Gm-action}, respectively. Then $(S,D)$ is projectively equivalent  to $(\{F=0\},\{F=H=0\})$ for $F$ and $H$ as in the third and fourth entries in the same row of Table \ref{tab:Gm-action}, respectively. In particular, any such pair $(S,D)$ is unique. Conversely, if $(S,D)$ is given by equations as in the third and fourth entries in a given row of Table \ref{tab:Gm-action}, then $(S, D)$ has singularities as in the first and second entries in the same row of Table \ref{tab:Gm-action} and $(S,D)$ is $\mathbb C^*$-invariant. Furthermore the element $\lambda\in \SL(4, \mathbb C^*)$, as defined in Lemma \ref{lemma:critical1PS}, given in the entry of the corresponding row of Table \ref{tab:Gm-action} is a generator of the $\mathbb C^*$-action. 
\end{lemma}
	
\begin{table}[!ht]%
\begin{center}
\begin{tabular}{|p{1.8cm}|p{3cm}|p{4cm}|c|c|}
\hline
$\mathrm{Sing}(S)$ & $\mathrm{Sing}(D)$ & $F$& $H$ & $\lambda$ \\
\hline\hline
$P_i=\boldsymbol{A}_2$, $i=1,2,3$ &  $\boldsymbol{A}_1$ at each $P_i$ & $x_0x_1x_3+x_2^3$ & $x_2$ & $\overline \lambda_2$\\
\hline
$P=\boldsymbol{A}_3$, $Q_1=\boldsymbol{A}_1$, $Q_2=\boldsymbol{A}_1$ & $D=2L+L'$, $Q_1,Q_2\in L$, $\mathrm{Sing}(S)\cap L'=\emptyset$ & $x_0x_1x_3+x_1x_2^2+x_0x_2^2$ & $x_3$ & $\overline\lambda_3$ \\
\hline
$P=\boldsymbol{A}_4$, $Q=\boldsymbol{A}_1$ & $\boldsymbol{A}_3$ at $Q$ & $x_0x_1x_3 + x_0x_2^2+x_1^2x_2$ & $x_3$ & $\lambda_5$ \\
\hline
$P=\boldsymbol{A}_5$, $Q=\boldsymbol{A}_1$ & $\boldsymbol{A}_2$ at $Q$ & $x_0x_2^2+x_0x_1x_3+x_1^3$ & $x_3$ & $\lambda_6 $ \\
\hline
$P=\boldsymbol{D}_4$ &  $\boldsymbol{D}_4$ not at $P$ & $x_0^2x_3 + x_1^3 +x_2^3 $ & $x_3$ & $\lambda_9 $ \\
\hline
$P=\boldsymbol{D}_5$ & $\boldsymbol{A}_3$ not at $P$ & $x_0^2x_3+x_0x_2^2+x_1^2x_2$ & $x_3$ & $\overline\lambda_6$ \\
\hline
$P={\boldsymbol{E}}_6$ & $\boldsymbol{A}_2$ not at $P$ & $x_0^2x_3+x_0x_2^2+x_1^3$ & $x_3$ & $ \overline\lambda_4$ \\
\hline
\end{tabular}
\caption{Some pairs $(S,D)$ invariant under a $\mathbb C^*$-action.}
\label{tab:Gm-action}
\end{center}
\end{table}
\begin{proof}
There is a unique surface $S$ with three $\boldsymbol{A}_2$ singularities \cite[p. 255]{classificationcubics} which corresponds to the equation in Table \ref{tab:Gm-action}. When a  surface $S$ has singularities $\boldsymbol{A}_4+\boldsymbol{A}_1$, $\boldsymbol{A}_5+\boldsymbol{A}_1$, $\boldsymbol{D}_4$, $\boldsymbol{D}_5$ or ${\boldsymbol{E}}_6$, and a $\mathbb C^*$-action, the equation for $F$ follows from \cite[{Table $3$}]{du2000hypersurfaces}. If $S$ has singularities $\boldsymbol{A}_3+2\boldsymbol{A}_1$, then \cite[Table $3$]{du2000hypersurfaces} gives that $S$ has equation $x_3f_2(x_0,x_1)+x^2_2l_1(x_0,x_1)=0$, where $x_0x_1$ has a non-zero coefficient in $f_2$, since otherwise $S$ is singular along a line. Hence, after a change of coordinates involving only variables $x_0$ and $x_1$ and rescaling $x_3$, we obtain the desired result. It is trivial to check that each one-parameter subgroup $\lambda$ in the corresponding row of Table \ref{tab:Gm-action} leaves $S$ invariant, and therefore $\lambda$ is a generator of the $\mathbb C^*$-action.

Given $H$, denote $D_H=\{F=H=0\}\subset S$. We need to show that for $(S,D)$ with prescribed singularities, $D_H=D$ if and only if $H$ is as stated in Table \ref{tab:Gm-action}. Verifying that for $F$ and $H$ as in the table, the pair $(S,D)$ has the exepected singularities is straight forward and we omit it. We verify the converse.

Suppose that $S$ has three $\boldsymbol{A}_2$ singularities. Then we may assume that $F=x_0x_1x_3+x_2^3$ and the singularities correspond to $P_1=(1,0,0,0)$, $P_2=(0,1,0,0)$ and $P_3=(1,0,0,0)$. There are only three lines $L_1,L_2,L_3$ in $S$ \cite[p. 255]{classificationcubics}, which correspond to $\{x_2=x_i=0\}$ for $i=0,1,3$, respectively. Clearly any two of these intersect at each of the points $P_j$. Moreover $D_H=D=\sum L_i$ and $D$ has an $\boldsymbol{A}_1$ singularity at each $P_i$, as stated in Table \ref{tab:Gm-action}.

Suppose that $S$ has an ${\boldsymbol{E}}_6$ singularity at a point $P$ and $D$ has an $\boldsymbol{A}_2$ singularity at a point $Q\neq P$ and $(S,D)$ is $\mathbb C^*$-invariant. Without loss of generality, we can now assume that $F=x_0^2x_3+x_0x_2^2+x_1^3$, $H=\sum a_ix_i$ for some parameters $a_i$ and $P=(0,0,0,1)$. Since $\overline \lambda_4$ is a generator of the $\mathbb C^*$-action, then $\overline\lambda_4(t)\cdot H=a_0t^{11}x_0+a_1t^3x_1+a_2t^{-1}x_2+a_3t^{-13}x_3$. Therefore $D_H$ is $\mathbb C^*$-invariant if and only if $H=x_i$ for some $i=0,\dots,3$. Notice that this happens every time the entries of $\lambda$ are distinct. If $H=x_0$, then $D_H$ is a triple line. If $H=x_1$, then $D_H$ is the union of a conic and a line, and therefore $D_H$ does not have an $\boldsymbol{A}_2$ singularity. If $H=x_2$, then $D_H$ has an $\boldsymbol{A}_2$ singularity at $P$. If $H=x_3$, then $D_H$ has an $\boldsymbol{A}_2$ singularity at $Q=(1,0,0,0)\neq P$ and $D_H=D$. 

Suppose $S$ has a $\boldsymbol{D}_5$ singularity at a point $P$, $D$ has an $\boldsymbol{A}_3$ singularity at a point $Q\neq P$ and $(S,D)$ is $\mathbb C^*$-invariant. There is a unique pair satisfying these conditions. Reasoning as in the previous case, we may assume $F=x_0^2x_3+x_0x_2^2+x_1^2x_2$, $H=x_i$ for some $i=0,\dots,3$ and $P=(0,0,0,1)$.  It follows from the equations that $\overline\lambda_6$ generates the $\mathbb C^*$-action. If $H=x_0$ or $H=x_2$, then the support of $D_H$ contains a double line. If $H=x_2$, then $D_H$ has an $\boldsymbol{A}_3$ singularity at $P$. If $H=x_3$, then $D_H$ has an $\boldsymbol{A}_3$ singularity at $Q=(1,0,0,0)\neq P$ and $D_H=D$. 

Suppose $S$ has an $\boldsymbol{A}_5$ singularity at a point $P$ and an $\boldsymbol{A}_1$ singularity at a point $Q$, $D$ has an $\boldsymbol{A}_2$ singularity at $Q$ and $(S,D)$ is $\mathbb C^*$-invariant. We may assume $\lambda_6$ generates the $\mathbb C^*$-action, $F=x_0x_2^2+x_0x_1x_3+x_1^3$, $H=x_i$ for some $i=0,\dots,3$, $P=(0,0,0,1)$ and $Q=(1,0,0,0)$. If $H=x_0$ then $D_H$ is a triple line. If $H=x_1$, then $D_H$ has a double line in its support. If $H=x_2$, then $D_H$ has two $\boldsymbol{A}_1$ singularities. If $H=x_3$, then $D_H$ has an $\boldsymbol{A}_2$ singularity at $Q=(1,0,0,0)\neq P$ and $D_H=D$. 

Suppose $S$ has an $\boldsymbol{A}_4$ singularity at a point $P$ and an $\boldsymbol{A}_1$ singularity at a point $Q$, $D$ has an $\boldsymbol{A}_3$ singularity at $Q$ and $(S,D)$ is $\mathbb C^*$-invariant. We may assume $\lambda_5$ generates the $\mathbb C^*$-action, $F=x_0x_1x_3 + x_0x_2^2+x_1^2x_2$, $H=x_i$ for some $i=0,\dots,3$, $P=(0,0,0,1)$ and $Q=(1,0,0,0)$. If $H=x_0$ or $H=x_1$ then $D_H$ contains a double line in its support. If $H=x_2$, then $D_H$ has three $\boldsymbol{A}_2$ singularities and if $H=x_3$, then $D_H$ has an $\boldsymbol{A}_2$ singularity at $Q$ and $D_H=D$. 

Suppose $S$ has a $\boldsymbol{D}_4$ singularity at a point $P$, $D$ has a $\boldsymbol{D}_4$ singularity at a point $Q\neq P$ and $(S,D)$ is $\mathbb C^*$-invariant. We may assume the generator of the $\mathbb C^*$-action is $\lambda_9$, $F=x_0^2x_3 + x_1^3 +x_2^3 $ and $P=(0,0,0,1)$. If $D_H$ is $\lambda_9$-invariant, either $H=x_i$ for some $i=0,\dots,3$ or $H=x_1-ax_2$ for $a\neq 0$. If $H=x_0$, then $D_H$ has a $\boldsymbol{D}_4$ singularity at $P$. If $H=x_1$ or $H=x_2$, then $D_H$ has an $\boldsymbol{A}_2$ singularity. If $H=x_1-ax_2$ with $a\neq 0$, then $D_H=\{x_0^2x_3+\left(1 + \frac{1}{a}\right)x_1^3=0, x_2=\frac{x_1}{a}\}$ has an $\boldsymbol{A}_2$ singularity. If $H=x_3$, then $D_H$ has a $\boldsymbol{D}_4$ singularity at $Q=(1,0,0,0)\neq P$ and $D_H=D$. 

Suppose $S$ has an $\boldsymbol{A}_3$ singularity at a point $P$, two $\boldsymbol{A}_1$ singularities at points $Q_1$ and $Q_2$, $D=2L+L'$ where $L$ is a line containing $Q_1$ and $Q_2$ and $L'$ is a line such that $P, Q_1,Q_2\not\in L'$. Furthermore, suppose $(S,D)$ is $\mathbb C^*$-invariant. We may assume that $\overline \lambda_3$ is the generator of the $\mathbb C^*$-action, $F=x_0x_1x_3+x_1x_2^2+x_0x_2^2$, $P=(0,0,0,1)$, $Q_1=(1,0,0,0)$, $Q_2=(0,1,0,0)$ and $L=\{x_2=x_3=0\}$. Moreover, if $D_H$ is $\overline\lambda_3$-invariant, either $H=x_i$ for some $i=0,\dots,3$ or $H=x_0-ax_1$ for $a\neq 0$. If $H=x_0$ or $H=x_1$, then $D_H$ does not contain $L$ in its support. If $H=x_2$ or $H=x_0-ax_1$, then $D_H$ is reduced. If $H=x_3$, then $D_H=2L+L'$, where $L'=\{x1+x_0=x_3=0\}$. Since $P, Q_1,Q_2\not\in L$, then $D_H=D$.

\end{proof}

\section{Proof of main theorems}\label{sec:stable}
We present the proofs of theorems \ref{theorem:stableGIT} and \ref{theorem:GITb}. First,  we reduce the amount of pairs we need to consider to those with isolated singularities:
\begin{lemma}
\label{lemma:realbad}
Let $(S,D)$ be a pair.
\begin{enumerate}
	\item If $S$ is reducible or not normal, then $(S,D)$ is $t$-unstable for $t\in [0,1)$.
	\item If $D$ is not reduced, then, $(S,D)$ is $t$-unstable for $t \in (1/5,1]$.
\end{enumerate} 
\end{lemma}
\begin{proof}
The case where $S$ is reducible follows from \cite[Theorem 1.3]{VGITour}. By Serre's criterion, any hypersurface of dimension $2$ is non-normal if and only if it has non-isolated singularities. The latter are classified for cubic surfaces in \cite[Case E]{classificationcubics}, hence $S$ is an irreducible non-normal cubic surface if and only if it is projectively equivalent to  
$\{x_3f_2(x_0,x_1)+f_3(x_0,x_1)+x_2g_2(x_0,x_1)=0\}.$
Then $\mu_t(S,D, \lambda_{10})\geqslant 1-t >0$. If $D$ is not reduced, we may assume $(S,D)$ is as in Lemma \ref{lemma:2LL}. Then $\mu_t( S,D, \lambda_3)= -1+5t >0$, if $t>\frac{1}{5}$.
\end{proof}

\begin{proof}[{Theorem \ref{theorem:stableGIT}}]
Let $(S,D)$ be a pair defined by equations $F$ and $H$. Notice that Lemma \ref{lemma:realbad} tells us that $S$ being normal is a necessary condition for $(S,D)$ to be $t$-stable for any $t\in(0,1)$. In particular $S$ has a finite number of singularities, since it is a surface. By Theorem \ref{theorem:unstMain}, the pair $(S,D)$ is $t$-stable if and only if for any $g\in \SL(4, \mathbb C)$ the monomials with non-zero coefficients of $(g\cdot F, g\cdot H)$ are not contained in a pair of sets $N^\oplus_t(\lambda, x_i)$ --- characterized geometrically in Section~\ref{sec:sing} --- which is maximal for every given $t$, as stated in Theorem \ref{theorem:unstMain}. These maximal sets can be found algorithmically \cite{VGITour, gallardo-jmg-code}. This is equivalent to the conditions in the statement. We verify the conditions for each $t\in (0,1)$. We will refer to the singularities of $D$ in terms of the ADE classification as in sections \ref{sec:sing} and \ref{sec:1ps}. These will be equivalent to the global description used in the statement of Theorem \ref{theorem:stableGIT} by Table \ref{tab:cubic-curves}.

Suppose $t\in(0,\frac{1}{5})$ and $(\lambda, x_i)=(\ol_3, x_3)$. Then $S$ cannot have an $\boldsymbol{A}_3$ singularity or a degeneration of one. When $(\lambda, x_i)=(\lambda_9, x_3)$, we deduce that $S$ cannot have a $\boldsymbol{D}_4$ singularity or a degeneration of one (this condition is redundant since $\boldsymbol{D}_4$ is a degeneration of $\boldsymbol{A}_3$). From $(\lambda, x_i)=(\lambda_1, x_2)$ or $(\lambda, x_i)=(\ol_2, x_2)$ we deduce that if $P\in D$ then $P$ is a singular point of $S$ of type at worst $\boldsymbol{A}_1$. We obtain the same condition if $(\lambda, x_i)=(\lambda_2, x_1)$. This completes the proof when $t\in (0, \frac{1}{5})$.

When $t=\frac{1}{5}$, the maximal sets $N^\oplus_t(\lambda, x_i)$ are the same as for $t\in\left(0, \frac{1}{5}\right)$ with the addition of $N^\oplus_t(\lambda_3, x_0)$, which represents the monomials of the equations of any pair $(S',D')$ such that $D'$ is not reduced. Therefore $(S,D)$ is $\frac{1}{5}$-stable if and only if in addition to the conditions for $t$-stability when $t\in(0,\frac{1}{5})$, $D$ is not reduced. Hence (ii) follows.

Let $t\in\left(\frac{1}{5}, \frac{1}{3}\right)$. The maximal $t$-non-stable sets $N^\oplus_t(\lambda, x_i)$ are the same as for $t=\frac{1}{5}$ but replacing the set $N^\oplus_t(\ol_3, x_3)$ with both $N^\oplus_t(\lambda_7,x_3)$ and $N^\oplus_t(\lambda_5,x_3)$. A pair $(S',D')$ whose defining equations have coefficients in one of $N^\oplus_t(\ol_3, x_3)$, $N^\oplus_t(\lambda_7,x_3)$ and $N^\oplus_t(\lambda_5,x_3)$ require that $S'$ has (a degeneration of) an $\boldsymbol{A}_3$ singularity, $S'$ is not normal or $S'$ has (a degeneration of) an $\boldsymbol{A}_4$ singularity, respectively. The second condition is redundant by Lemma \ref{lemma:realbad}. Hence a $t$-stable pair $(S,D)$ may now have $\boldsymbol{A}_3$ singularities but not $\boldsymbol{A}_4$ singularities. However, the coefficients of the equations of $(S,D)$ cannot be in $N^\oplus_t(\lambda_9, x_3)$ and hence $S$ cannot have (degenerations of) $\boldsymbol{D}_4$ singularities.  Therefore $(S,D)$  is $t$-stable if and only if $S$ has at worst $\boldsymbol{A}_3$ singularities, $D$ is reduced and if $D$ supports a surface singularity $P$, then $P$ must be an $\boldsymbol{A}_1$-singularity and (iii) follows.

Let $t=\frac{1}{3}$. The maximal sets $N^\oplus_t(\lambda, x_i)$ are the same as for $t\in\left(\frac{1}{5}, \frac{1}{3}\right)$ with the addition of $N^\oplus_t(\lambda_5, x_0)$, which represents the monomials of the equations of any pair $(S',D')$ such that $D'$ has (a degeneration of) an $\boldsymbol{A}_3$ singularity at a singular point $P$ of $S$. Hence $(S,D)$ is $\frac{1}{3}$-stable if and only if it is $t$-stable for $t\in\left(\frac{1}{5}, \frac{1}{3}\right)$ but $D$ does not have (a degeneration of) an $\boldsymbol{A}_3$ singularity at a singular point of $P$. Hence (iv) follows.

Let $t\in \left(\frac{1}{3},\frac{3}{7}\right)$. The maximal sets are $N^\oplus_t(\lambda, x_i)$ the same as for $t=\frac{1}{3}$ but replacing the set $N^\oplus_t(\lambda_5, x_3)$ --- parametrizing pairs $(S',D')$ where $S'$ has (a  degeneration of) an $\boldsymbol{A}_4$ singularity --- with the set $N^\oplus_t(\lambda_6,x_3)$ --- parametrizing pairs $(S',D')$ where $S'$ has (a degeneration of) an $\boldsymbol{A}_5$ singularity. Hence a $t$-stable pair $(S,D)$ may now have $\boldsymbol{A}_4$ singularities but not $\boldsymbol{A}_5$ singularities. However, the coefficients of the equations of $(S,D)$ cannot be in $N^\oplus_t(\lambda_9, x_3)$ and hence $S$ cannot have (degenerations of) $\boldsymbol{D}_4$ singularities.  Furthermore the restrictions for $t=\frac{1}{3}$ regarding $D$ still apply. Therefore a pair $(S,D)$ is $t$-stable if and only if satisfies the conditions in (v).

Let $t=\frac{3}{7}$. The maximal sets $N^\oplus_t(\lambda, x_i)$ are the same as for $t\in\left(\frac{1}{3}, \frac{3}{7}\right)$ but replacing the set $N^\oplus_t(\lambda_5, x_0)$ --- parametrizing pairs $(S',D')$ such that $D'$ has (a degeneration of) an $\boldsymbol{A}_3$ singularity at a surface singularity of $S'$ ---,  for both the set $N^\oplus_t(\ol_6, x_0)$ --- parametrizing pairs $(S',D')$ such that $D'$ has (a degeneration of) an $\boldsymbol{A}_2$ singularity at a surface singularity of $S'$ --- and the set $N^\oplus_t(\ol_9, x_0)$ --- parametrizing pairs $(S',D')$ such that $D'$ has (a degeneration of) an $\boldsymbol{A}_4$ singularity. Hence (vi) follows.

Let $t\in \left(\frac{3}{7}, \frac{5}{9}\right]$. The difference between the maximal sets for $N^\oplus_t(\lambda, x_i)$ and for $N^\oplus_{\frac{3}{7}}(\lambda, x_i)$ consists of three new sets ($N^\oplus_t(\ol_6, x_3)$, $N^\oplus_t(\lambda_8, x_3)$ and $N^\oplus_t(\lambda_{10}, x_3)$) and three sets that do not appear for $t$ anymore ($N^\oplus_t(\lambda_9, x_3)$, $N^\oplus_t(\lambda_6, x_3)$, $N^\oplus_t(\lambda_7, x_3)$). The three new sets parametrize pairs $(S',D')$ such that $S'$ has at least either (a degeneration of) one $\boldsymbol{D}_5$ singularity , a degeneration of one $\tilde{{\boldsymbol{E}}_6}$ singularity or one line of singularities, respectively. The three sets that are not maximal non-stable sets for $t$ parametrize pairs $(S',D')$ such that $S'$ has (a degeneration of) a $\boldsymbol{D}_4$ singularity, an $\boldsymbol{A}_5$ singularity and a line of singularities, respectively. Hence, the only difference with respect to $t=\frac{3}{7}$ is that we include pairs $(S,D)$ such that $S$ has at worst $\boldsymbol{A}_5$ or $\boldsymbol{D}_4$ singularities and (vii) follows.

Let $t=\frac{5}{9}$. The difference between the maximal sets for $N^\oplus_t(\lambda, x_i)$ for $t\in \left(\frac{3}{7},\frac{5}{9}\right)$ and for $N^\oplus_{\frac{5}{9}}(\lambda, x_i)$ consists of replacing the set $N^\oplus_t(\lambda_3, x_0)$ --- parametrizing pairs $(S',D')$ such that $D'$ is non-reduced --- for the set\linebreak $N^\oplus_t(\lambda_6, x_0)$ --- parametrizing pairs $(S',D')$ such that $D'$ has (a degeneration of) an $\boldsymbol{A}_3$ singularity. Hence a $\frac{5}{9}$-stable pair $(S,D)$ is a $t$-stable pair for $t\in \left(\frac{3}{7},\frac{5}{9}\right)$ such that $D$ has at worst an $\boldsymbol{A}_2$ singularity. Notice that $D$ is still reduced by Lemma \ref{lemma:realbad}. Hence (viii) follows.

Let $t\in\left(\frac{5}{9}, \frac{9}{13}\right)$. The difference between the maximal sets for $N^\oplus_t(\lambda, x_i)$ for $t\in \left(\frac{5}{9},\frac{9}{13}\right)$ and for $N^\oplus_{\frac{5}{9}}(\lambda, x_i)$ consists of replacing the set $N^\oplus_t(\ol_6, x_3)$ --- parametrizing pairs $(S',D')$ such that $S'$ has (a degeneration of) a $\boldsymbol{D}_5$ singularity --- for the set $N^\oplus_t(\ol_4, x_3)$ --- parametrizing pairs $(S',D')$ such that $S'$ has (a degeneration of) an ${\boldsymbol{E}}_6$ singularity. Hence (ix) follows.

Let $t=\frac{9}{13}$. The difference between the maximal sets for $N^\oplus_t(\lambda, x_i)$ for $t\in \left(\frac{5}{9}, \frac{9}{13}\right)$ and for $N^\oplus_{\frac{9}{13}}(\lambda, x_i)$ consists of replacing the set $N^\oplus_t(\ol_6, x_0)$ --- parametrizing pairs $(S',D')$ such that $D'$ has (a degeneration of) an $\boldsymbol{A}_2$ singularity at a singular point of $S'$ ---, the set $N^\oplus_t(\ol_9, x_0)$ --- parametrizing pairs $(S',D')$ such that $D'$ has (a degeneration of) a $\boldsymbol{D}_4$ singularity --- and the set $N^\oplus_t(\lambda_6, x_0)$ --- parametrizing pairs $(S',D')$ such that $D'$ has (a degeneration of) an $\boldsymbol{A}_3$ singularity --- for the set $N^\oplus_t(\lambda_4, x_0)$ --- parametrizing pairs $(S',D')$ such that $D'$ has (a degeneration of) an $\boldsymbol{A}_2$ singularity. Hence (x) follows.

Let $t\in \left(\frac{9}{13}, 1\right)$. The maximal sets $N^\oplus_t(\lambda, x_i)$ are the same as for $N^\oplus_{\frac{9}{13}}(\lambda, x_i)$ but removing the set $N^\oplus_t(\ol_4, x_3)$, which parametrizes pairs $(S',D')$ where $S'$ has an ${\boldsymbol{E}}_6$ singularities. Hence such surfaces are now $t$-stable providing they do not violate any other conditions. This concludes the proof of the theorem.
\end{proof}

\begin{proof}[{Theorem \ref{theorem:GITb}}]
Suppose $(S,D)$ --- defined by polynomials $F$ and $H$ --- belongs to a closed strictly $t$-semistable orbit. By Lemma \ref{lemma:Gm-action}, they are generated by monomials in $N^0_t(\lambda, x_i)$ for some $(\lambda, x_i)$ such that $N^\oplus_t(x_i,\lambda)$ is maximal with respect to the containment of order of sets. Since there is a finite number of $\lambda$ to consider (those in Lemma \ref{lemma:critical1PS}), this is a finite computation which can be carried out by software \cite{VGITour, gallardo-jmg-code}. 
For each pair $(\lambda, x_i)$, there is a change of coordinates that gives a natural bijection between $N^0(\lambda, x_i)$ and $N^0(\overline \lambda, x_{3-i})$. Therefore about half of the values are redundant and we have two possible choices for each $F$ and $H$ if $t\neq t_1,\ldots, t_5$ three choices if $t=t_1, t_2, t_4,t_5$ and four if $t=t_3$.

Similarly, by  \cite[Lemma 3.2]{VGITour} and Lemma \ref{lemma:critical1PS} we can check that the pair $(\overline S, \overline D)$ corresponding to $\overline F=x_0x_3x_1+x_2^3$, $\overline H=x_2$ is strictly $t$-semistable. Suppose that $(\lambda, x_i)=(\lambda_1, x_2)$. Then $F=x_0x_3f_1(x_1,x_2)+f_3(x_1,x_2)$ and $H=g_1(x_1,x_2)$. After a change of variables involving only $x_1$ and $x_2$, we may assume that $F=x_0x_3x_1+f_3(x_1,x_2)$. We will show that the closure of $(S,D)$ contains $(\overline S, \overline D)$. Let $\gamma=\mathrm{Diag}(1, 1, 0, -2)$ be a one-parameter subgroup. Then
$$\lim_{t\rightarrow 0}\gamma(t)\cdot F=x_0x_1x_3+bx_2^3 \text{ and }\lim_{t\rightarrow 0}\gamma(t)\cdot H=x_2.$$
If $b=0$, then $\lim_{t\rightarrow 0}\gamma(t)\cdot S$ is reducible, which is impossible as it is not $t$-stable for any value of $t\in (0,1)$ by Lemma \ref{lemma:realbad}. Therefore $b\neq 0$ and by rescaling we see that $\lim_{t\rightarrow 0}\gamma(t)\cdot (S,D)=(\overline S,\overline D)$. Hence, the closure of the orbit of $(S,D)$ contains $(\overline S, \overline D)$, which we tackle next.

Suppose that $(\lambda, x_i)=(\lambda_2, x_1)$. Then $F=x_1^3+x_0f_2(x_2,x_3)$ and $H=x_1$. After a change of variables involving only $x_2$ and $x_3$ we may assume that $F=x_1^3+x_0x_2x_3$. We can do similar changes of variables in the rest of the cases and end up with $F$ and $H$ not depending on any parameters. Observe that since $(S,D)$ is strictly $t$-semistable, the stabilizer subgroup of $(S,D)$, namely $G_{(S,D)}\subset \SL(4,\mathbb C)$ is infinite (see \cite[Remark 8.1 (5)]{dolgachev}). In particular there is a $\mathbb C^*$-action on $(S,D)$.
Lemma \ref{lemma:Gm-action} classifies the singularities of $(S,D)$ uniquely according to their equations. For each $t\in (0,1)$, the proof of Theorem \ref{theorem:GITb} follows once we recall the classification of plane cubic curves according to their isolated singularities (see Table \ref{tab:cubic-curves}).
\end{proof}

\bibliographystyle{amsalpha}
\bibliography{quintic} 	

\newcommand{\etalchar}[1]{$^{#1}$}
\providecommand{\bysame}{\leavevmode\hbox to3em{\hrulefill}\thinspace}
\providecommand{\MR}{\relax\ifhmode\unskip\space\fi MR }
\providecommand{\MRhref}[2]{%
  \href{http://www.ams.org/mathscinet-getitem?mr=#1}{#2}
}
\providecommand{\href}[2]{#2}
\begin{thebibliography}{GMGS18}

\bibitem[ACT02]{allcock2002complex}
Daniel Allcock, James~A Carlson, and Domingo Toledo, \emph{The complex
  hyperbolic geometry of the moduli space of cubic surfaces}, J. Algebraic
  Geom. \textbf{11} (2002), no.~4, 659--724.

\bibitem[Arn75]{ArnoldLong}
V.~I. Arnold, \emph{Critical points of smooth functions, and their normal
  forms}, Uspehi Mat. Nauk \textbf{30} (1975), no.~5(185), 3--65. \MR{0420689}

\bibitem[Arn76]{arnoldinvetiones}
V.I. Arnold, \emph{Local normal forms of functions}, Inventiones Mathematicae
  \textbf{35} (1976), no.~1, 87--109.

\bibitem[BW79]{classificationcubics}
J.W Bruce and C.~Wall, \emph{On the classification of cubic surfaces}, J.
  London Math. Soc. (2) \textbf{2} (1979), no.~2, 245--256.

\bibitem[DH98]{dolgachev1998variation}
Igor~V. Dolgachev and Yi~Hu, \emph{Variation of geometric invariant theory
  quotients}, Inst. Hautes \'{E}tudes Sci. Publ. Math. (1998), no.~87, 5--56,
  With an appendix by Nicolas Ressayre. \MR{1659282}

\bibitem[Dol03]{dolgachev}
Igor Dolgachev, \emph{Lectures on invariant theory}, London Mathematical
  Society Lecture Note Series, vol. 296, Cambridge University Press, Cambridge,
  2003. \MR{2004511}

\bibitem[DPW00]{du2000hypersurfaces}
AA~Du~Plessis and CTC Wall, \emph{Hypersurfaces in {$\mathbb P^n(\mathbb C)$}
  with one-parameter symmetry groups}, R. Soc. Lond. Proc. Ser. A Math. Phys.
  Eng. Sci. \textbf{456} (2000), no.~2002, 2515--2541. \MR{1796494}

\bibitem[DT92]{Ding-Tian-YTDconjecture}
Wei~Yue Ding and Gang Tian, \emph{K\"ahler-{E}instein metrics and the
  generalized {F}utaki invariant}, Invent. Math. \textbf{110} (1992), no.~2,
  315--335. \MR{1185586 (93m:53039)}

\bibitem[GMG17]{gallardo-jmg-code}
P.~Gallardo and J.~Martinez-Garcia, \emph{Variations of {GIT} quotients package
  v0.6.13.}, {https://doi.org/10.15125/BATH-00458}, 2017.

\bibitem[GMG18]{VGITour}
\bysame, \emph{Variations of geometric invariant quotients for pairs, a
  computational approach}, Proc. Amer. Math. Soc. \textbf{146} (2018), no.~6,
  2395--2408. \MR{3778143}

\bibitem[GMGS18]{gallardo-margar-spotti}
Patricio Gallardo, Jesus Martinez-Garcia, and Cristiano Spotti,
  \emph{Applications of the moduli continuity method to log {K}-stable pairs},
  arXiv preprint arXiv:1811.00088 (2018).

\bibitem[Hil93]{hilbert}
David Hilbert, \emph{{\"U}ber die vollen invariantensysteme}, Math. Ann.
  \textbf{42} (1893), no.~3, 313--373.

\bibitem[HKT09]{hacking2009stable}
Paul Hacking, Sean Keel, and Jenia Tevelev, \emph{Stable pair, tropical, and
  log canonical compactifications of moduli spaces of del {P}ezzo surfaces},
  Invent. Math. \textbf{178} (2009), no.~1, 173--227. \MR{2534095}

\bibitem[HP10]{Hacking-Prokhorov}
Paul Hacking and Yuri Prokhorov, \emph{Smoothable del {P}ezzo surfaces with
  quotient singularities}, Compos. Math. \textbf{146} (2010), no.~1, 169--192.
  \MR{2581246 (2011f:14062)}

\bibitem[I{\etalchar{+}}82]{ishii1982moduli}
Shihoko Ishii et~al., \emph{Moduli space of polarized del {P}ezzo surfaces and
  its compactification}, Tokyo J. Math \textbf{5} (1982), no.~2, 289--297.

\bibitem[Laz09]{laza2009deformations}
Radu Laza, \emph{Deformations of singularities and variation of {GIT}
  quotients}, Trans. Amer. Math. Soc. \textbf{361} (2009), no.~4, 2109--2161.
  \MR{2465831}

\bibitem[Loo77]{Looijenga-Root-Systems}
Eduard Looijenga, \emph{Root systems and elliptic curves}, Invent. Math.
  \textbf{38} (1976/77), no.~1, 17--32. \MR{0466134}

\bibitem[LPZ19]{Laza-Pearlstein-Zheng}
Radu Laza, Gregory Pearlstein, and Zheng Zhang, \emph{On the moduli space of
  pairs consisting of a cubic threefold and a hyperplane}, to appear in Adv.
  Math. (arXiv:1710.08056).

\bibitem[Muk03]{mukai}
Shigeru Mukai, \emph{An introduction to invariants and moduli}, Cambridge
  Studies in Advanced Mathematics, vol.~81, Cambridge University Press,
  Cambridge, 2003. \MR{2004218}

\bibitem[Nar82]{naruki1982cross}
Isao Naruki, \emph{Cross ratio variety as a moduli space of cubic surfaces},
  Proceedings of the London Mathematical Society \textbf{3} (1982), no.~1,
  1--30.

\bibitem[OSS16]{Odaka-Spotti-Sun}
Yuji Odaka, Cristiano Spotti, and Song Sun, \emph{Compact moduli spaces of del
  {P}ezzo surfaces and {K}\"ahler-{E}instein metrics}, J. Differential Geom.
  \textbf{102} (2016), no.~1, 127--172. \MR{3447088}

\bibitem[Pin78]{pinkham1978deformations}
Henry Pinkham, \emph{Deformations of normal surface singularities with
  {$\mathbb C^*$} action}, Math. Ann. \textbf{232} (1978), no.~1, 65--84.

\bibitem[ST99]{shustin1999versal}
Eugenii Shustin and Ilya Tyomkin, \emph{Versal deformation of algebraic
  hypersurfaces with isolated singularities}, Math. Ann. \textbf{313} (1999),
  no.~2, 297--314.

\bibitem[Tha96]{thaddeus1996geometric}
Michael Thaddeus, \emph{Geometric invariant theory and flips}, J. Amer. Math.
  Soc. \textbf{9} (1996), no.~3, 691--723. \MR{1333296}

\end{thebibliography}
	
\end{document}